\documentclass{amsart}
\usepackage{amsmath,amssymb,epsfig,amscd,amsthm}

\usepackage{verbatim}

\numberwithin{equation}{section}

\newtheorem{lemma}[equation]{Lemma}
\newtheorem{thm}[equation]{Theorem}

\newtheorem{cor}[equation]{Corollary}
\newtheorem{prop}[equation]{Proposition}

\theoremstyle{remark}

\newtheorem{remark}[equation]{Remark}

\newtheorem*{acknowledgments}{Acknowledgments}

\renewcommand{\bar}[1]{#1\llap{$\overline{\phantom{\rm#1}}$}}
\newcommand{\Aberk}[1]{\ensuremath{\bA_{\mathrm{Berk},#1}^1}}

\newcommand{\lra}{\longrightarrow}

\DeclareMathOperator{\tor}{{tor}}
\DeclareMathOperator{\End}{{End}}

\DeclareMathOperator{\hhat}{{\widehat{h}}}
\DeclareMathOperator{\supp}{{Supp}}
\DeclareMathOperator{\gal}{{Gal}}
\newcommand{\bA}{{\mathbb A}}
\newcommand{\N}{{\mathbb N}}
\newcommand{\Z}{{\mathbb Z}}

\newcommand{\R}{{\mathbb R}}
\newcommand{\C}{{\mathbb C}}
\newcommand{\E}{{\mathbb E}}
\newcommand{\M}{{\mathbb M}}

\newcommand{\Fp}{{\mathbb{F}_p}}
\newcommand{\Fq}{{\mathbb{F}_q}}
\newcommand{\Fqs}{{\mathbb{F}_{q^s}}}
\newcommand{\Fpbar}{{\bar{\Fp}}}
\newcommand{\Fqbar}{{\bar{\Fq}}}

\newcommand{\Kbar}{{\bar{K}}}

\newcommand{\bG}{{\mathbb G}}
\newcommand{\Gal}{{\rm Gal}}

\newcommand{\cV}{\mathcal{V}}

\newcommand{\into}{\hookrightarrow}

\newcommand{\cD}{\mathcal{D}}

\renewcommand{\l}{\lambda}

\renewcommand{\a}{\alpha}

\newcommand{\g}{\gamma}
\renewcommand{\d}{\delta}
\newcommand{\e}{\epsilon}

\renewcommand{\th}{\theta}

\newcommand{\bfa}{{\mathbf a}}
\newcommand{\bfb}{{\mathbf b}}
\newcommand{\bfc}{{\mathbf c}}

\newcommand{\bff}{{\mathbf f}}

\newcommand{\bP}{{\mathbb P}}

\newcommand{\helpme}[1]{{\sf $\heartsuit\heartsuit$ Meta-remark: [#1]}}
\newcommand{\fixme}[1]{{\sf $\spadesuit\spadesuit$ Meta-remark: [#1]}}

\begin{document}



\title{Torsion points in families of Drinfeld modules}

\author{D.~Ghioca}
\address{
Dragos Ghioca\\
Department of Mathematics\\
University of British Columbia\\
Vancouver, BC V6T 1Z2\\
Canada
}
\email{dghioca@math.ubc.ca}

\author{L.-C.~Hsia}
\address{
Liang-Chung Hsia\\
Department of Mathematics\\ 
National Taiwan Normal  University\\
Taiwan, ROC
}
\email{hsia@math.ntnu.edu.tw}

\keywords {Drinfeld modules; torsion points}
\subjclass[2010]{Primary 37P05; Secondary 37P10}
\thanks{The first author was partially supported by an NSERC Discovery Grant. The second
  author was partially supported by the National Center of Theoretical Sciences of Taiwan and 
 NSC Grant  99-2115-M-003-012-MY3.}


 \begin{abstract}
 Let $\Phi^\l$ be an algebraic family of Drinfeld modules defined over a field $K$ of characteristic $p$, and let $\bfa,\bfb\in K[\l]$. Assume that neither $\bfa(\l)$ nor $\bfb(\l)$ is a torsion point for $\Phi^\l$ for all $\l$. If there exist infinitely many $\l\in\Kbar$ such that both $\bfa(\l)$ and $\bfb(\l)$ are torsion points for $\Phi^\l$, then we show that for each $\l\in\Kbar$, we have that $\bfa(\l)$ is torsion for $\Phi^\l$ if and only if $\bfb(\l)$ is torsion for $\Phi^\l$. In the case $\bfa,\bfb\in K$, then we prove in addition that $\bfa$ and $\bfb$ must be $\Fpbar$-linearly dependent.
 \end{abstract}

\date{\today}

\maketitle

\section{Introduction}
\label{intro}

Lang \cite{Lang-int} proved that if a curve $C\subset \bA^2$ contains infinitely many points whose coordinates are both roots of unity, then $C$ is the zero set of an equation of the form $X^mY^n=\alpha$, for some $m,n\in\Z$ (not both equal to $0$) and some root of unity $\alpha$. In this case we note that if $C$ projects dominantly on both axis, then for \emph{each} point $(x,y)\in C$ we know that $x$ is a root of unity \emph{if and only if} $y$ is a root of unity. In particular, the following result is a corollary of Lang's Theorem.
\begin{cor}
\label{corollary Lang}
Let $F_1,F_2\in \C(z)$ be nonconstant rational maps such that there exist infinitely many $\l\in\C$ such that both $F_1(\l)$ and $F_2(\l)$ are roots of unity. Then $F_1$ and $F_2$ are multiplicatively dependent, and therefore for each $\l\in\C$, $F_1(\l)$ is a root of unity if and only if $F_2(\l)$ is a root of unity. 
\end{cor}

Lang's result is a special case of the Manin-Mumford Conjecture  (proven by Raynaud \cite{Ray1, Ray2} for abelian varieties, and by Hindry \cite{Hindry} for semiabelian varieties). The Manin-Mumford Conjecture (in its most general form asked by Lang) predicts that the set of
torsion points of a semiabelian variety $G$ defined over $\C$ is not
Zariski dense in a subvariety $V$ of $G$, unless $V$ is a translate of
an algebraic subgroup of $G$ by a torsion point.  Pink and Zilber have
suggested extending the Manin-Mumford conjecture to a more general question
regarding unlikely intersections between a subvariety $V$ of a
semiabelian scheme $G$ and algebraic subgroups of the fibers of $G$ having
codimension greater than the dimension of $V$ (see \cite{BMZ,
  Habegger, M-Z-1, M-Z-2, Pink}). 

Masser and Zannier~ \cite{M-Z-1, M-Z-2} study the Pink-Zilber Conjecture when $G$ is the square of
the Legendre family of elliptic curves. We state below  a special case of their theorem.
\begin{thm}[Masser-Zannier] 
For each $\l\in\C\setminus\{0,1\}$, let $E_\l$ be the elliptic curve
given by the equation $Y^2=X(X-1)(X-\l)$. Let $P_\l$ and $Q_\l  $
be two given families of points on $E_\l$ depending algebraically on the parameter $\l$.
Suppose there exist infinitely
many $\l\in\C$ such that both $P_\l$ and $Q_\l$ are torsion points for
$E_\l$. Then the points $P_\l$ and $Q_\l$  are linearly dependent over
$\Z$ on the generic fiber of the elliptic surface $E_\l$.
\end{thm}

Both Lang's result and Masser-Zannier's result are special cases of the following problem. Let $\{G_\l\}$ be an algebraic family of algebraic groups, and let $\{P_\l\}$ and $\{Q_\l\}$ be two algebraic families of points on $G_\l$. If there exist infinitely many $\l$ such that both $P_\l$ and $Q_\l$ are torsion points, then at least one of the following three properties holds:
\begin{enumerate}
\item $P_\l$ is torsion for all $\l$.
\item $Q_\l$ is torsion for all $\l$.
\item for each $\l$, $P_\l$ is a torsion point if and only if
  $Q_\l$ is a torsion point. 
\end{enumerate}

In the arithmetic theory of function fields of positive characteristic, {\em
  Drinfeld modules} play a  role similar to elliptic curves over 
number fields. Hence, it is natural to ask if
there exist results for family of Drinfeld mdoules that are parallel
to those obtained by Masser and Zannier. 
In this paper we study the first instance of the above problem in
characteristic $p$, where $G_\l:=(\bG_a,\Phi^\l)$ is the constant
family of additive group schemes endowed with the action of a
Drinfeld module $\Phi^\l$ (which belongs to an algebraic family of
such Drinfeld modules). More precisely, let $K$ be a field extension
of $\Fq(t)$ and let $K(z)$ be the rational function field in
variable $z.$ Let $\Phi : \Fq[t] \to \End_{K(z)}(\bG_a)$ be  a Drinfeld
module defined over $K(z)$ (See Section~\ref{statement of results} for
details). Equivalently, $\Phi$ is also regarded as a family of Drinfeld
modules by specialization. That is, by letting $z=\l$ for $\l\in
\Kbar$ we obtain an algebraic family of Drinfeld modules, denoted by
$\Phi^\l$ which are defined over $\Kbar$. Then we are able to prove 
that the above trichotomy must hold. The following result is a special
case of our more general result (Theorem~\ref{main result}). 
   
\begin{thm}
\label{nice main result}
Let $r\ge 2$ be an integer, let $q$ be a power of a prime number $p$,
let $K$ be a field extension of $\Fq(t)$, and let $g_1,\dots,g_{r-1}\in K[z]$. We
let $\Phi:\Fq[t]\lra \End_{K(z)}(\bG_a)  $ be the family of Drinfeld modules defined by
$$\Phi_t(x)=tx+\sum_{i=1}^{r-1}g_i(z) x^{q^i} +x^{q^r}.$$
Let $\bfa,\bfb\in K[z]$ and assume there exist infinitely many $\l\in\Kbar$ such that both $\bfa(\l)$ and $\bfb(\l)$ are torsion points for $\Phi^\l$. Then at least one of the following properties holds:
\begin{enumerate} 
\item 
$\bfa$ is a torsion point for $\Phi$.
\item 
$\bfb$ is a torsion point for $\Phi$.
\item for each $\l\in\Kbar$, $\bfa(\l)$ is a torsion point for $\Phi^\l$ if and only if $\bfb(\l)$ is a torsion point for $\Phi^\l$.
\end{enumerate}
\end{thm}

For certain families of Drinfeld modules and for constant starting points $\bfa,\bfb$ (i.e., no dependence on $z$) we can give an explicit condition of when there exist infinitely many $\l\in\Kbar$ such that both $\bfa$ and $\bfb$ are torsion points for $\Phi^\l$.

\begin{thm}
\label{nice corollary}
Let $r\ge 2$ be an integer and let $K$ be a field  extension of
$\Fq(t)$ such that $\Fq$ is algebraically closed in $K$. Let
$\bfa,\bfb\in K$ and let $\Phi:\Fq[t]\lra \End_{K(z)}(\bG_a)$ be
the family of Drinfeld modules given by 
$$\Phi_t(x)=tx+z x^q + x^{q^r}.$$
Then there exist infinitely many $\l\in\Kbar$ such that both $\bfa$
and $\bfb$ are torsion points for $\Phi^\l$, if and only if $\bfa$ and
$\bfb$ are linearly dependent over $\Fq$. 
\end{thm}

Going in the opposite direction, when $K$ is a constant field
extension of $\Fq(t)$ we can give a more precise relation between
$\bfa$ and $\bfb.$ 

\begin{thm}
\label{nice corollary 2}
Let $r\ge 2$ be an integer, let $q$ be a power of a prime number $p$, let  $\bfa,\bfb\in K = \Fqbar(t)$ and let $\Phi:\Fq[t]\lra \End(\bG_a)$ be the family of Drinfeld modules given by
$$\Phi_t(x)=tx+z x^{q} + x^{q^r}.$$
Then there exist infinitely many $\l\in\Kbar$ such that both $\bfa$ and $\bfb$ are torsion points for $\Phi^\l$, if and only if $\bfa$ and $\bfb$ are linearly dependent over $\Fq$.
\end{thm}

Both Theorems \ref{nice corollary} and \ref{nice corollary 2} should be true without any extra condition on the (constant) starting points $\bfa$ and $\bfb$. However, our methods cannot be extended to this general case. In order to prove Theorems~\ref{nice corollary} and \ref{nice corollary 2} we employ a finer analysis of the valuation of torsion points for Drinfeld modules which fails for arbitrary extension fields $K$ of $\Fq(t)$.

A question analogous to the problem studied by Masser and Zannier
in the setting of arithmetic dynamics is the following.  Given
complex numbers $a$ and $b$ and an integer $d\ge 2$, when do there
exist infinitely many $\lambda\in\C$ such that both $a$ and $b$ are
preperiodic for the action of $f_{\lambda}(x):=x^d+\lambda$ on $\C$?
This question was first raised by  Zannier at an American Institute of
Mathematics workshop in 2008 and studied by Baker and DeMarco in~
\cite{Matt-Laura}.  They show that this happens if  and only if
$a^d=b^d$.  In \cite{prep}, the authors together with Tucker extended
Baker-DeMarco's result to general families of polynomials. On the
other hand, a Drinfeld module can be viewed as a
collection of (additive) polynomials acting on the affine line - so we can
apply the techniques and results that have been developed in
arithmetic dynamics  into the study of  Drinfeld modules,
especially to Diophantine problems arising from Drinfeld modules; this point of view
is proven to be useful in our paper. The problem studied in this paper naturally can be viewed as an
analogue of the problems from~\cite{Matt-Laura} and~\cite{prep} in the
setting of Drinfeld modules. Hence, our method for studing  the
question raised above follows closely the methods used
in~\cite{Matt-Laura}. We will apply the techniques used in~\cite{prep}
to our situation. The reader will find materials in
Section~\ref{spaces}   and results in Section~\ref{preliminaries}
parallel to~\cite[Sect. 4-6]{prep}. However, as we work in the world of
positive characteristic, we can not expect that all the techniques
developed in~\cite{Matt-Laura} and~\cite{prep} can be applied to our
situation.  This paper reveals the differences and
difficulties in the study of the same questions in the setting of
Drinfeld modules. In particular, we cannot use complex analysis in order to derive the explicit relations between $\bfa$ and $\bfb$ as in Theorems~\ref{nice corollary} and \ref{nice corollary 2} - instead we use arguments from valuation theory which are amenable to Drinfeld modules.

The plan of our paper is as follows. 
In Section~\ref{statement of results} we set up our notation and state our main result (Theorem~\ref{main result}), and then describe the method of our proof. 
In Section~\ref{spaces} we give a brief 
overview of Berkovich spaces. Then, in Section~\ref{preliminaries} we
compute  the capacities of the generalized $v$-adic Mandelbrot sets
associated to a generic point $\bfc$  for a family of Drinfeld
modules.  We proceed with our proof of Theorem~\ref{main result} in Section~\ref{proof
  of our main result}.  Then, in
Section~\ref{conclusion} we conclude our paper by proving
Theorems~\ref{nice main result}, \ref{nice corollary} and \ref{nice corollary 2}.

\begin{acknowledgments} 
  The
  authors thank Thomas Tucker for several useful
  conversations.
\end{acknowledgments}

\section{Preliminary and statement of the main results}
\label{statement of results}

In this section, we give a brief review of the theory of Drinfeld
modules and height functions that is relevant to our discussion below. 
Throughout the paper, we let $q$ be a power of a prime $p$ and let
$\Fq$ be a finite field of $q$ elements.

\subsection{Drinfeld Modules}
\label{subsec:drinfeld modules}

We first define Drinfeld modules (of generic characteristic) -- for
more details, see \cite{Goss}. Let $L$ be a field extension of the
rational function field $\Fq(t)$.  
A Drinfeld $\Fq[t]$-module 
defined over $L$ is an $\Fq$-algebra homomorphism $\Phi:\Fq[t]\lra \End_L(\bG_a)$ such that 
$$\Phi_t(x):=tx+\sum_{i=1}^{r-1}a_i x^{q^i} + x^{q^r} \text{, with }
a_i \in L \quad\text{for all $i = 1, \ldots, r-1$}$$  
where $\Phi_a$ denotes the image of $a\in \Fq[t]$ under $\Phi$
and $r\ge 1$ is called the
rank of $\Phi$. Note that as an $\Fq$-algebra homomorphism, $\Phi$ is
uniquely determined by the 
action of $\Phi_t$ on $\bG_a$. In general, $\Phi_t(x)$ is not required to
be monic in $x$. However, at the expense of replacing $\Phi$ by a Drinfeld
module which is conjugated to it, i.e.  
\begin{equation}
\label{gamma conjugation}
\Psi_t(x)=\gamma^{-1}\Phi_t(\gamma x),
\end{equation} 
for a suitable $\gamma\in\bar{L}$ we obtain that $\Psi_t$ is
monic. Note that $\Psi$ is defined over $L(\gamma)$ which is a 
finite field extension of $L.$

A point $x\in\bar{L}$ is called \emph{torsion} for $\Phi$ if there
exists a nonzero $a\in \Fq[t]$ such that $x$ is in the kernel of
$\Phi_a$. We denote by $\Phi_{\tor}$ the set of all torsion
points for $\Phi$. It is immediate to see that $x$ is torsion if and
only if its orbit under the action of $\Phi_t$ is finite, i.e. $x$ is
preperiodic for the map $\Phi_t$. Note that if $\Psi$ is a Drinfeld
module conjugated to $\Phi$ as in \eqref{gamma conjugation}, then
$x\in\Phi_{\tor}$ if and only if $\gamma^{-1}x\in \Psi_{\tor}$.

\subsection{Canonical heights}
\label{subsect:canonical height}

Let $K$ be a finitely generated transcendental extension of $\Fq$. At
the expense of replacing $K$ by a finite extension and replacing $\Fq$
with its algebraic closure in this finite extension, we may assume
there exists a smooth projective, geometrically irreducible variety
$\cV$ defined over $\Fq$ whose function field is $K$ (see
\cite{deJong}). We let 
then $\Omega_K$ be the set of places of $K$ corresponding to the
codimension one irreducible subvarieties of $\cV$. Then for each
$v\in\Omega_K$ there exists a 
positive integer $N_v$ such that for all $\alpha\in K^{\ast}$ we have
$\prod_{v\in \Omega}\,|\alpha|_v^{N_v} = 1$ where for $v\in \Omega_K$, the
corresponding absolute value is denoted by $|\cdot |_v$ (for more details see \cite[\S~2.3]{lang} or \cite[\S~1.4.6]{bg06}).   We note that if $x\in K$ is a $v$-adic unit for all $v\in\Omega_K$ then $x\in\Fq$.

Let $\C_v$ be a fixed completion of the algebraic closure of a
completion $K_v$ of $(K,|\cdot |_v)$. 
Note that there is a unique extension of $|\cdot |_v$ to an absolute
value on $\C_v$. By abuse of notation, we still denote this
extension by $|\cdot|_v$. Let $\Phi$ be a Drinfeld module of rank $r$
defined over $\C_v$.  Following Poonen \cite{Poonen} and Wang \cite{Wang},
 for each  $x\in \C_v$, the  \emph{local canonical height} of $x$ is
 defined as follows   
$$\hhat_{\Phi,v}(x):=\lim_{n\to\infty}
\frac{\log^+|\Phi_{t^n}(x)|_v}{q^{rn}},$$
where by $\log^+z$ we always denote $\log\max\{z,1\}$ (for any real
number $z$).  It is immediate that
$\hhat_{\Phi,v}(\Phi_{t^i}(x))=q^{ir} \hhat_{\Phi,v}(x)$
and thus $\hhat_{\Phi,v}(x)=0$ whenever $x\in\Phi_{\tor}$.

Now, if $f(x)=\sum_{i=0}^d a_i x^i$ is any polynomial defined over $\C_v$, then
$|f(x)|_v=|a_dx^d|>|x|_v$ when $|x|_v>r_v$, where
\begin{equation}
\label{definition N_v}
r_v=r_v(f):=\max\left\{1,\max\left\{
    \left|\frac{a_i}{a_d}\right|^{\frac{1}{d-i}}\right\}_{0\le
    i<d}\right\}.
\end{equation}
Moreover, for a Drinfeld module $\Phi$, if $|x|_v>r_v(\Phi_t)$
then $\hhat_{\Phi,v}(x)=\log|x|_v>0$.  For more details see \cite{equi}  
and \cite{L-C}.

We fix an algebraic closure $\Kbar$ of $K$, and for each
$v\in\Omega_K$ we fix an embedding $\Kbar\into\C_v$.
The \emph{global canonical height} $\hhat_{\Phi}(x)$ associated to the
Drinfeld module $\Phi$ was first introduced by Denis~\cite{denis92}
(Denis defined the global canonical heights for general $T$-modules
which are higher dimensional analogue of Drinfeld modules). 
For each $x\in\Kbar$, the global canonical height is defined  as
$$\hhat_\Phi(x)=\lim_{n\to\infty} \frac{h\left(\Phi_{t^n}(x)\right)}{q^{rn}},$$
where $h$ is the usual (logarithmic) Weil height on $\Kbar$.  As shown
in \cite{Poonen} and \cite{Wang}, the global canonical height
decomposes into a sum of the corresponding local canonical
heights. Furthermore, $\hhat_\Phi(x)=0$ if and only if
$x\in\Phi_{\tor}$ (see \cite[Proposition 3 (v)]{Wang}).

\begin{remark}
(1). We note that the theory of canonical height associated
to a Drinfeld module is a special case of the canonical heights
associated to morphisms on algebraic varieties developed by
Call~and~Silverman (see~\cite{Call-Silverman} for details). 
\\
(2) The definition for the canonical height functions given above
seems to depend on the particular choice of the map $\Phi_t$. On the
other hand, one can define the canonical heights
$\hhat_{\Phi}$  as in~\cite{denis92} by letting 
$$\hhat_\Phi(x)=\lim_{\deg(R)\to\infty}
\frac{h\left(\Phi_{R}(x)\right)}{q^{r\deg(R)}},$$
and similar formula for canonical local heights $\hhat_{\Phi,v}(x)$
where $R$ runs through all non-constant polynomials in $\Fq[t].$ In \cite{Poonen} and \cite{Wang} it is proven that  both definitions yield the same height
function. 
\end{remark}

Let $\Phi : \Fq[t] \to \End_L(\bG_a)$ be the Drinfeld module as in
Theorem~\ref{nice main result} with $L = K(z).$ Then, our main result is
the following.

\begin{thm}
\label{main result}
Let let $K$ be a field extension of $\Fq(t)$, $r\ge 2$ be an integer,
and let $g_1,\dots,g_{r-1}\in K[z]$. We let
$\Phi:\Fq[t]\lra \End_{K(z)}(\bG_a)$ be the family of Drinfeld 
modules defined by  
$$\Phi_t(x)=tx+\sum_{i=1}^{r-1}g_i(z) x^{q^i} +x^{q^r}.$$
Let $\bfa,\bfb\in K[z]$ and assume the following inequality holds:
\begin{equation}
\label{condition on degrees of starting points}
\min\{\deg(\bfa),\deg(\bfb)\}> \max\left\{\frac{\deg(g_1)}{q^r-q},\dots, \frac{\deg(g_{r-1})}{q^r-q^{r-1}}\right\}.
\end{equation}
If there exist infinitely many $\l\in\Kbar$ such that both $\bfa(\l)$
and $\bfb(\l)$ are torsion points for $\Phi^\l$, then for each
$\l\in\Kbar$ we have that $\bfa(\l)$ is torsion for $\Phi^\l$ if  and
only if $\bfb(\l)$ is torsion for $\Phi^\l$. Furthermore, in this
case, if $C_\bfa$ and $C_\bfb$ are the leading coefficients of $\bfa$,
respectively of $\bfb$, then
$\frac{C_\bfa^{\deg(\bfb)}}{C_\bfb^{\deg(\bfa)}}\in\Fpbar$. 
\end{thm}
\noindent Note that inequality \eqref{condition on degrees of starting points}
prevents $\bfa$ and $\bfb$ to be torsion points for \emph{all} $\l$
(see Proposition~\ref{lem:degree in l}).  

\begin{remark}
One can consider a more general Drinfeld module such that the additive
group $\bG_a$ is equipped with an action by a ring $A$ cosisting of
all regular functions on a smooth projective curve defined over $\Fq$
with a fixed point removed. In our definition above, the curve is
$\bP^1$ with the distinguished point at infinity. However, we lose no
generality by restricting ourselves to the case $A=\Fq[t]$ since
always $\Fq[t]$ embeds into such a ring $A$, and moreover, the notion
of torsion is identical for $\Phi$ and for its restriction to $\Fq[t]$. 
\end{remark}

Our results and proofs are inspired by the results of
\cite{Matt-Laura} so that the strategy for the proof of Theorem~\ref{main
  result} essentially follows the ideas in the paper \cite{Matt-Laura}. 
However there are different  technical details in our proofs. Also,
Drinfeld modules are better vehicle for the method from
\cite{prep} which generalizes \cite{Matt-Laura}.

Following~\cite{Matt-Laura}, the strategy to prove Theorem~\ref{main
  result} is the following. For the family of Drinfeld modules $\Phi$, 
We define the $v$-adic generalized Mandelbrot sets $M_{\bfc,v}$
(inside the Berkovich affine line) associated to  any $\bfc\in K[z]$
similar to that introduced in \cite{Matt-Laura}. 
We then  use the equidistribution result discovered independently by
Baker-Rumely~\cite{Baker-Rumely06}, Chambert-Loir~\cite{CL} and
Favre-Rivera-Letelier~\cite{favre-rivera, favre-rivera06} to deduce that the corresponding
$v$-adic Mandelbrot sets $M_{\bfa,v}$ and $M_{\bfb,v}$ for the given
points $\bfa$ and $\bfb$ are equal for each $v\in \Omega_K$. Here we
use the version obtained 
by Baker and Rumely~\cite{Baker-Rumely06} to connect the
equidistribution result to arithmetic capacities which is used in the
final step of the proof to show that the leading coeffients of $\bfa$
and $\bfb$ have the desired relation. 

In order to apply the equidistribution results and capacity theory
over nonarchimedean fields mentioned above, we need to
introduce the Berkovich space associated to the affine line which
we give an overview in the next section.



\section{Equidistribution Theorem and Berkovich spaces}
\label{spaces}


As mentioned above, we will need to apply the arithmetic
equidistribution discovered independently by Baker-Rumely,
Chambert-Loir and Favre-Rivera-Letelier. When the base field is a
nonarchimedean field, the equidistribution theorem is best stated over
the Berkovich space associated to the underlying variety in question. 
For the convenience of the reader, we give a review on the
Berkovich spaces in this section. 

We use the  version of equidistribtution theorem obtained
by Baker and Rumely which connects the equidistribution result to the
theory of arithmetic capacities. Hence, the material presented in this
section is mainly from the book~\cite{Baker-Rumely} by Baker and
Rumely.   
The main body of this Section is taken from~\cite[Section 4]{prep}
which is written according to the summary in~\cite{Matt-Laura}. For
more details on the arithmetic 
equidistribution as well as a detailed introduction to
Berkovich line we refer the reader to the book~\cite{Baker-Rumely}. 


Let $K$ be a field of characteristic $p$ endowed with a product
formula, and let $\Omega_K$ be the set of its inequivalent absolute
values. For each $v\in\Omega_K$, we let $\C_v$ be the completion of an
algebraic closure of the completion of $K$ at $v$.   
Let $\Aberk{\C_v}$ denote the Berkovich affine  line over
$\C_v$ (see \cite{berkovich} or  \cite[\S~2.1]{Baker-Rumely} for
details). Then $\Aberk{\C_v}$ is a 
locally compact, Hausdorff, path-connected space containing $\C_v$ as a
dense subspace (with the topology induced from the $v$-adic absolute
value). As a topological space, $\Aberk{\C_v}$ is the set consisting of
all multiplicative seminorms, denoted by $[\cdot]_x,$ on $\C_v[T]$ extending
the absolute value $|\cdot|_v$ on $\C_v$ endowed with the weakest
topology such that the map $z\mapsto [f]_z$ is continuous for all
$f\in \C_v[T]$. 
The set of seminorms can be described as follows. If
$\{D(a_i,r_i)\}_i$ is any decreasing nested sequence of closed disks
$D(c_i,r_i)$ centered at points $c_i\in \C_v$ of radius $r_i\ge 0$, then
the map $f\mapsto \lim_{i\to\infty}\,[f]_{D(c_i,r_i)}$ defines a
multiplicative seminorm on $\C_v[T]$ where $[f]_{D(c_i,r_i)}$ is the
sup-norm of $f$ over the closed disk 
$D(a_i, r_i).$ Berkovich's classification theorem says that  
there are exactly four types of points, Type I, II, III and IV. 
The first three types of points can be described in terms of closed
disks $\zeta = D(c,r) = \cap D(c_i,r_i)$ where $c\in \C_v$ and $r\ge 0.$
The corresponding 
multiplicative seminorm is just $f \mapsto [f]_{D(c,r)}$ for $f\in
\C_v[T].$ Then, $\zeta$ is of Type I, II or III  if and only if $r= 0,
r\in |\C_v^{\ast}|_v$ or $r\not\in |\C_v^{\ast}|_v$ respectively. As for
Type IV points, they correspond to sequences of decreasing nested
disks $D(c_i,r_i)$ such that $\cap D(c_i,r_i) = \emptyset$ and the
multiplicative seminorm is $f\mapsto \lim_{i\to \infty}
[f]_{D(c_i,r_i)}$ as described above.  For details, see
\cite{berkovich} or \cite{Baker-Rumely}. 
For $\zeta\in \Aberk{\C_v},$ we sometimes write $|\zeta|_v$
instead of $[T]_{\zeta}.$ 

In order to apply the main equidistribution result from \cite[Theorem
7.52]{Baker-Rumely}, we recall the potential theory on the 
affine line over $\C_v$.   The right setting for nonarchimedean potential theory is the
potential theory on $\Aberk{\C_v}$ developed in \cite{Baker-Rumely}. We
quote part of a nice summary of the theory 
from \cite[\S~2.2~and~2.3]{Matt-Laura} without going into details. We
refer the reader to \cite{Matt-Laura, Baker-Rumely}  for all the
details and proofs. Let $E$ be a compact subset of
$\Aberk{\C_v}.$ Then analogous to the complex case, the logarithmic
capacity $\g(E) = e^{-V(E)}$ and the Green's function $G_E$ of $E$
relative to $\infty$ can be defined where $V(E)$ is the infimum of the 
\emph{energy integral} with respect to all possible probability
measures supported on $E$. More precisely,
$$V(E)=\inf_{\mu}\int\int_{E\times E} -\log\delta(x,y) d\mu(x)d\mu(y),$$
where the infimum is computed with respect to all probability measures $\mu$ supported on $E$, while for $x, y\in \Aberk{\C_v},$  the function $\delta(x,y)$ 
is the \emph{Hsia kernel} (see \cite[Proposition~4.1]{Baker-Rumely}):
$$\delta(x,y):=\limsup_{\substack{z,w\in \C_v\\ z\to x, w\to y}}\mid z-w\mid _v .$$ 
The following are basic properties of the logarithmic 
capacity of $E$. 
\begin{itemize}
\item
  If $E_1, E_2$ are two compact subsets of $\Aberk{\C_v}$ such that
  $E_1\subset E_2$ then $\g(E_1) \le \g(E_2).$

\item
  If $E = \{\zeta\}$ where $\zeta$ is a Type II or III point
  corresponding to   a closed disk $D(c,r)$ then $\g(E) = r > 0.$ 
  \cite[Example~6.3]{Baker-Rumely}. (This can be viewed as analogue of
  the fact that a closed disk $D(c,r)$ of positive radius $r$ in $\C_v$
  has logarithmic capacity $\g(D(c,r)) = r$.)    
\end{itemize}
If $\g(E) > 0,$ then the exists a unique probability measure $\mu_E$
attaining the infimum of the energy integral. Furthermore, the support of
$\mu_E$ is contained in the boundary of the unbounded component of
$\Aberk{\C_v}\setminus E.$  

The Green's function 
$G_E(z)$ of $E$ relative to infinity is a well-defined nonnegative
real-valued subharmonic function on $\Aberk{\C_v}$ which is harmonic on $\Aberk{\C_v}\setminus E$ (in the sense of
\cite[Chapter~8]{Baker-Rumely}). 
If $\g(E)=0$, then there exists no Green's function associated to the set $E$. Indeed, as shown in \cite[Proposition 7.17, page 151]{Baker-Rumely}, if $\gamma(\partial E)=0$ then  there exists no nonconstant harmonic function on $\Aberk{\C_v}\setminus E$ which is bounded below (this is the Strong Maximum Principle for harmonic functions defined on Berkovich spaces). 
The following result is \cite[Lemma~2.2~and~2.5]{Matt-Laura}, and it gives a characterization of
the Green's function of the set $E$.  

\begin{lemma}
  \label{green function}
Let $E$ be a compact subset of $\Aberk{\C_v}$ and let $U$ be the unbounded component of
$\Aberk{\C_v}\setminus E$. 

\begin{itemize}
\item[(1)]
If $\g(E)>0$ (i.e. $V(E)<\infty$), then
  $G_E(z) = V(E) + \log |z|_v $ for all $z\in \Aberk{\C_v}$ such that
  $|z|_v$ is sufficiently large.

\item[(2)]
If $G_E(z) = 0$ for all $z\in E,$ then $G_E$ is continuous on
  $\Aberk{\C_v},$ $\supp(\mu_E) = \partial U$ and $G_E(z) > 0$ if and
  only if $z\in U.$  

\item[(3)]
  If $G :\Aberk{\C_v}\to \R$ is a continuous subharmonic function which
  is harmonic on $U,$ identically zero on $E,$ and such that $G(z) -
  \log^+|z|_v$ is bounded, then $G = G_E$. Furthermore, if $G(z)=\log|z|_v + V + o(1)$ (as $|z|_v\to\infty$) for some $V<\infty$, then $V(E)=V$ and so, $\g(E)=e^{-V}$.

\end{itemize}

\end{lemma}

To state the equidistribution result from \cite{Baker-Rumely}, we
consider the compact \emph{Berkovich ad\`elic sets} which are of the
following form
$$
\E := \prod_{v\in \Omega}\, E_v
$$
where $E_v$ is a non-empty compact subset of $\Aberk{\C_v}$ for each
$v\in \Omega$ and where $E_v$ is the closed unit disk $\cD(0,1)$ in
$\Aberk{\C_v}$ for all but finitely many $v\in \Omega.$ The
\emph{logarithmic capacity} $\g(\E)$ of $\E$ is defined as follows
$$\g(\E) = \prod_{v\in \Omega}\,\g(E_v)^{N_v}, $$
where the positive integers $N_v$ are the ones associated to the product formula on the global field $K$. 
Note that this is a finite product as for all but finitely many 
$v\in\Omega,$ $\g(E_v) = \g(\cD(0,1)) = 1.$   Let $G_v = G_{E_v}$  be
the Green's function of $E_v$ 
relative to $\infty$ for each $v\in \Omega.$ For every $v\in\Omega,$ we
fix an embedding $\Kbar\into\C_v.$ Let 
$S\subset \Kbar$ be any finite subset that is invariant under the action
of the Galois group 
$\gal(\Kbar/K)$. We define the height $h_{\E}(S)$ of $S$ relative to $\E$
by
\begin{equation}
\label{def ade hig}
h_{\E}(S) = \sum_{v\in\Omega} N_v\left(\frac{1}{|S|}\sum_{z\in S}G_v(z)\right).
\end{equation}
Note that this definition is independent of any particular embedding
$\Kbar\into\C_v$ that we choose at $v\in \Omega.$
The following is a special case of the equidistribution result 
\cite[Theorem~7.52]{Baker-Rumely} that we need for our application.
\begin{thm}
 \label{thm:equidistribution}
Let $\E = \prod_{v\in \Omega} E_v$ be a compact Berkovich ad\`elic set
 with $\g(\E)=1.$ Suppose that 
 $S_n$ is a sequence of $\gal(\Kbar/K)$-invariant finite subsets of
 $\Kbar$ with $|S_n|\to \infty$ and $h_{\E}(S_n) \to 0$ as $n\to
 \infty.$ 
For each $v\in\Omega$ and for each $n$ let $\d_n$ be the discrete
 probability measure supported equally on the elements of $S_n.$ Then the
 sequence of measures $\{\d_n\}$ converges weakly to $\mu_v$ the
 equilibrium measure on $E_v.$
\end{thm}

\section{The dynamics of the Drinfeld module family $\Phi$}  
\label{preliminaries}

Let $K$ be a finitely generated field extension of $\Fq(t)$. We work
with a family of Drinfeld modules $\Phi$ 
as given  in Theorem~\ref{main result}, i.e.
$$\Phi_t(x)=tx+  \sum_{i=1}^{r-1} g_i(z) x^{q^i}+x^{q^r},  $$
with $g_i(z) \in K[z]$ for $i = 1, \ldots, r-1$.  For convenience, we
let $g_0(z):=t$ be the corresponding constant polynomial.  
Let $\bfc \in K[z]$ be given. We define $f_{\bfc,n}(z):=
\Phi_{t^n}(\bfc)$ for each 
$n\in\N$. Note that $f_{\bfc,n}$ is a polynomial in $z$ with
coefficients in $K.$ 
Assume $m := \deg(\bfc)$ satisfies  the inequality
\begin{equation}
\label{what we need for degree}
m=\deg(\bfc)>\max_{i=0}^{r-1}\frac{\deg(g_i)}{q^r-q^i}.
\end{equation} 
We let $C_m$ be the leading coefficient of $\bfc$. 
In the next Lemma we compute the degrees of all polynomials $f_{\bfc,n}$ for all positive integers $n$.

\begin{lemma}
\label{lem:degree in l}
With the above hypothesis, the polynomial $f_{\bfc,n}(z)$ has degree
$m\cdot q^{rn}$ and leading coefficient $C_m^{q^{rn}}$ for each
$n\in\N$. 
\end{lemma}

\begin{proof}
The assertion  follows
easily by induction on $n$, using \eqref{what we need for degree}, since the term of highest degree in $z$ from $f_{\bfc,n}(z)$ is  $\bfc^{q^{rn}}$.  
\end{proof}

We immediately obtain as a corollary of Lemma~\ref{lem:degree in l} the fact that $\bfc$ is not preperiodic for $\bff$.  Furthermore we obtain that if $\bfc(\l)\in\Phi^\l_{\tor}$, then $\l\in\Kbar$.

Fix a place $v\in \Omega_K.$ Our first task  is to define the
{\em generalized Mandelbrot set} $M_{\bfc,v}$ associated to $\bfc$ and
establish that $M_{\bfc,v}$ is a compact subset of $\Aberk{\C_v}$. Roughly
speaking, $M_{\bfc,v}$ is the subset of $\C_v$ consisting of all $\l\in
\C_v$ such that $\bfc(\l)$ is $v$-adic bounded under the action of
$\Fq[t]$ with respect to the Drinfeld module structure $\Phi^\l$. A
simple observation is that the orbit of $\bfc(\l)$ under the action of
$\Fq[t]$ is $v$-adic bounded if and only if under the action of
$\Phi_t^\l$ the orbit $\{\Phi_{t^n}^\l(\bfc(\l))\mid n=0,1,2,\ldots\}$
is $v$-adic bounded.  Hence, in our definition for the 
generalized Mandelbrot set $M_{\bfc,v}$ below we shall only consider the
orbit of $\bfc(\l)$ under the action of $\Phi_t^\l.$ 
Following \cite{Matt-Laura, prep},  we put 
\[
 M_{\bfc,v} := \{\l \in \Aberk{\C_v} : \sup_n \left[f_{\bfc,n}(T)\right]_\l < \infty\}.
\]
Let $\l\in \C_v$ and recall the local canonical height
$\hhat_{\Phi^\l,v}(x)$ of $x\in \C_v$  is given by the formula 
$$
\hhat_{\Phi^{\l},v}(x) =\lim_{n\to\infty}
\frac{\log^+|\Phi_{t^n}^\l(x)|_v}{q^{rn}}.
$$
Notice that $\hhat_{\Phi^\l,v}(x)$ is a continuous function of both $\l$
and $x$. 
As $\C_v$ is a dense subspace of $\Aberk{\C_v}$,
continuity in $\l$ implies that  the canonical local height function 
$\hhat_{\Phi^\l,v}(\bfc(\lambda))$ has a natural extension on $\Aberk{\C_v}$ (note that the topology on
$\C_v$ is the restriction of the weak topology on $\Aberk{\C_v}$, so any
continuous function
on $\C_v$ will automatically have a unique extension to $\Aberk{\C_v}$). In the
following, we will extend $\hhat_{\Phi^\l,v}(\bfc(\lambda))$ to a
function of $\l$ on $\Aberk{\C_v}$ and view it as a continuous function on
$\Aberk{\C_v}.$ It follows from the definition of $M_{\bfc,v}$
that   
$\l\in M_{\bfc,v}$ if and only if $\hhat_{\Phi^\l,v}(\bfc(\lambda)) = 0.$  Thus, 
$M_{\bfc,v}$ is a closed subset of $\Aberk{\C_v}.$ In fact, the following is
true. 

\begin{prop}
\label{prop:bounded mandelbrot} 
$M_{\bfc,v}$ is a compact subset of $\Aberk{\C_v}.$ 
\end{prop}

We already showed that $M_{\bfc,v}$ is a closed subset of the locally
compact space $\Aberk{\C_v}$, and thus in order to prove
Proposition~\ref{prop:bounded mandelbrot} we only need to show that
$M_{\bfc,v}$ is a bounded subset of $\Aberk{\C_v}$. 

\begin{lemma}
\label{M_a is bounded for archimedean v}
$M_{\bfc,v}$ is a bounded subset of  $\Aberk{\C_v}$.
\end{lemma}

\begin{proof}
For each $i=0,\dots, r-1$, we let $D_i$ be the leading coefficient of $g_i$ (recall that $g_0=t$); also, let $d_i:=\deg(g_i)$. Using an argument identical as in deriving \eqref{definition N_v}, we see that there exists $M>1$ such that if $|\l|_v>M$, then 
$$|g_i(\l)|_v=|D_i|_v\cdot |\l|_v^{d_i}\text{ for each }i=0,\dots, r-1\text{, and}$$
$$|\bfc(\l)|_v=|C_m|_v\cdot |\l|_v^m>|\l|_v>1.$$
At the expense of replacing $M$ with a larger number we may assume that
$$M^{m(q^r-q^i)-d_i}\ge |D_i|_v\cdot |C_m|_v^{q^i-q^r}\text{ for each }i=0,\dots,r-1.$$
Note that we may achieve the above inequality for large $M$ since (by our assumption \eqref{what we need for degree}), $m>d_i/(q^r-q^i)$ for each $i=0,\dots,r-1$. Therefore, if $|\l|_v>M$, then
$$|\bfc(\l)|_v^{q^r}>\left|g_i(\l)\cdot \bfc(\l)^{q^i}\right|_v\text{ for each } i=0,\dots, r-1.$$
Hence, 
$$|\Phi^\l_t(\bfc(\l))|_v=|\bfc(\l)|_v^{q^r}>|\bfc(\l)|_v>|\l|_v>1.$$
This allows us to conclude that if $|\l|_v>M$, then
$|\Phi^\l_{t^n}(\bfc(\l)|_v\to\infty$ as $n\to\infty$. Thus $\l\notin
M_{\bfc,v}$ if $|\l|_v>M$. 
\end{proof}

 Next our goal is to compute
 the logarithmic capacities of the $v$-adic generalized 
 Mandelbrot sets $M_{\bfc,v}$ associated to $\bfc$ for the given
 family of Drinfeld module $\Phi.$

\begin{thm}
  \label{mandelbrot capacity}
The logarithmic capacity of
$M_{\bfc,v}$ is $\g(M_{\bfc,v}) = |C_m|_v^{-1/m} $. 
\end{thm}

The strategy for the proof of Theorem~\ref{mandelbrot capacity} is to
construct a continuous subharmonic function $G_{c,v} : \Aberk{\C_v}\to \R$
satisfying Lemma~\ref{green function}~(3).  
We let 

\begin{equation}
  \label{mandelbrot green function}
G_{\bfc,v}(\l) := \lim_{n\to \infty}\, \frac{1}{\deg(f_{\bfc,n})} \log^+ [f_{\bfc,n}(T)]_\l. 
\end{equation}
Then by a similar reasoning as in the proof of
\cite[Prop.~3.7]{Matt-Laura}, it can be shown that the limit exists
for all $\l\in \Aberk{\C_v}$.  
In fact, by the definition of canonical local height, for $\l\in \C_v$
we have
\begin{align*}
  G_{\bfc,v}(\l) & = \lim_{n\to \infty}\, \frac{1}{mq^{rn}}\log^+ |f_\l^{n}(\bfc(\l))|_v \quad \text{since 
  $\deg(f_{\bfc,n}) = m q^{rn}$ by Lemma~\ref{lem:degree in l}, } \\
&  = \frac{1}{m}\cdot \hhat_{\Phi^{\l},v}(\bfc(\l)) \quad \text{by the definition of canonical local height.}
\end{align*}

Note that $G_{\bfc,v}(\l) \ge 0$ for all $\l\in \Aberk{\C_v}.$
Moreover, by definition 
we see that $\l\in M_{\bfc,v}$ if and only if $G_{\bfc,v}(\l) = 0.$ 

\begin{lemma}
\label{green's function for mandelbrot}
$G_{\bfc,v}$ is the Green's function for $M_{\bfc,v}$ relative to $\infty.$ 
\end{lemma}

The proof is essentially the same as the proof of
\cite[Prop.~3.7]{Matt-Laura}, we simply give a sketch of the idea. 

\begin{proof}[Proof of Lemma~\ref{green's function for mandelbrot}.]
So, using the same argument as in the proof of
  \cite[Prop.~1.2]{BH88}, we observe that as a function of $\l,$ 
the
  function $\displaystyle \frac{\log^+[f_{\bfc,n}(T)]_\l}{\deg(f_{\bfc,n})}$
  converges uniformly on  compact subsets of $\Aberk{\C_v}$.
So, the function $\displaystyle \frac{\log^+[f_{\bfc,n}(T)]_\l}{\deg(f_{\bfc,n})}$
is a continuous subharmonic function on $\Aberk{\C_v}$,
which converges to $G_{\bfc,v}$ uniformly; hence it follows from
  \cite[Prop.~8.26(c)]{Baker-Rumely} that $G_{\bfc,v}$ is continuous and
  subharmonic on $\Aberk{\C_v}.$  Furthermore, as remarked above,
  $G_{\bfc,v}$ is zero on $M_{\bfc,v}.$ 

Arguing as in the proof of Lemma~\ref{M_a is bounded for archimedean v}, if $|\lambda|_v$ is sufficiently large, then for $n\ge 1$ we have
$$|f_{\bfc,n}(\l)|_v=|\Phi^{\lambda}_{t^n}(\bfc(\l))|_v=|C_m\lambda^{m}|_v^{q^{rn}}.$$
Hence, for $|\lambda|_v$ sufficiently large we have
  \begin{align*}
    G_{\bfc,v}(\l) & = \lim_{n\to \infty}\, \frac{1}{mq^{rn}} \log
    |f_{\bfc,n}(\l)|_v
    \\
     & = \log|\l|_v + \frac{\log |C_m|_v}{m}.
  \end{align*}
  It follows from Lemma~\ref{green function}~(3), that $G_{\bfc,v}$ is indeed
  the Green's function of $M_{\bfc,v}.$ 
\end{proof}

Now we are ready to prove Theorem~\ref{mandelbrot capacity}.
\begin{proof}[Proof of  Theorem~\ref{mandelbrot capacity}.]
As in the proof of Lemma~\ref{green's function for mandelbrot}, we
have
\[
 G_{\bfc,v}(\l)  = \log|\l|_v + \frac{\log |C_m|_v}{m} 
\]
for $|\l|_v$ sufficiently large. By Lemma~\ref{green
    function}~(3), we find that $V(M_{\bfc,v}) = \frac{\log |C_m|_v}{m}$. Hence,
  the logarithmic capacity of $M_{\bfc,v}$ is
  \[
  \g(M_{\bfc,v}) = e^{-V(M_{\bfc,v})} = \frac{1}{|C_m|^{1/m}_v}
\]
as desired. 
\end{proof}

Let $\M_{\bfc} = \prod_{v\in
\Omega} M_{\bfc,v}$ be the generalized ad\`elic Mandelbrot set associated to
$c$. As a corollary to Theorem~\ref{mandelbrot capacity} we see that
$\M_{\bfc}$ satisfies the hypothesis of Theorem~\ref{thm:equidistribution}.
\begin{cor}
 \label{adelic mandelbrot}
For all but finitely many nonarchimedean places $v$, we have that $M_{\bfc,v}$ is the closed unit disk $\cD(0;1)$ in $\Aberk{\C_v}$; furthermore $\g(\M_{\bfc}) = 1$.
\end{cor}

\begin{proof}
For each place $v$ where all coefficients of $g_i(z), i=0,
\ldots,r-1$ and of $\bfc(z)$ are
$v$-adic integral, and moreover $|C_m|_v=1$, we have that
$M_{\bfc,v}=\cD(0,1)$. Indeed, $\cD(0,1)\subset M_{\bfc,v}$ since then
$\Phi^{\lambda}_{t^n}(\bfc(\l))$ is always a $v$-adic integer. For the converse
implication we note that each coefficient of $f_{\bfc,n}(z)$ is a
$v$-adic integer, while the leading coefficient is a $v$-adic unit for
all $n\ge 1$; thus
$|f_{\bfc,n}(\lambda)|_v=|\lambda|_v^{mq^{rn}}\to\infty$ if
$|\lambda|_v>1$. Note that $C_m\ne 0$ and so, the second assertion in
Corollary~\ref{adelic mandelbrot} follows immediately by the product formula in $K$.
\end{proof} 

Using  the decomposition of the global canonical height as a sum of local canonical heights we obtain the following result.
\begin{cor}
\label{second important remark}
Let $\l\in\Kbar$, let $S$ be the set of $\Gal(\Kbar/K)$-conjugates of $\l$, and let $h_{\M_{\bfc}}$ be defined as in \eqref{def ade hig}. Then $\deg(\bfc)\cdot h_{\M_{\bfc}}(S)=\hhat_{\Phi^{\l}}(\bfc(\l))$.
\end{cor}

\begin{remark}
  \label{remark:second important remark}
Corollary~\ref{second important remark} is a result on specialization
of height functions relating the canonical heights of $\bfc(\l)$ in the
family and the height of the parameter $\l.$ As a consequence of
Corollary~\ref{second important remark}, we find that $\bfc(\l)$ is a
torsion point for $\Phi^\l$ if and only if $h_{\M_{\bfc}}(S) = 0.$ 

By abuse of notation we will use the notation $h_{\M_{\bfc}}(\l):=h_{\M_{\bfc}}(S)$ where $S$ is the $\Gal(\Kbar/K)$-orbit of $\l$.  
\end{remark}


\section{Proof of Theorem~\ref{main result}}
\label{proof of our main result}


We work under the hypothesis of Theorem~\ref{main result}, and we continue with the notation from the previous Sections. 

Recall that 
$\Phi_t(x)= tx + \sum_{i=0}^{r-1} g_i(z) x^{q^i}+x^{q^r}$ where we require that
$g_i\in K[z]$ for $i = 1, \ldots, r-1.$ Let $\bfa,
\bfb \in K[z]$ satisfying the hypothesis \eqref{condition on degrees
  of starting points} of Theorem~\ref{main result}. In particular,
both $\bfa$ and $\bfb$ have positive degrees.   At the expense
of replacing $K$ with the extension of $\Fq(t)$ generated by all
coefficients of $g_i$, and of $\bfa$ and of $\bfb$, we may assume that
$K$ is finitely generated over $\Fq(t)$.  
Let $\Omega_K$ be the set of all
inequivalent absolute values on $K$ constructed as
in Subsection~\ref{subsect:canonical height}. 

Next, assume there exist infinitely many $\lambda$ such that
$\bfa(\l), \bfb(\l)\in \Phi_{\tor}^{\lambda}$.  As a consequence of
Lemma~\ref{lem:degree in l}, we have $\l\in\Kbar$.

Let $h_{\M_{\bfa}}(z)$  
($h_{\M_{\bfb}}(z)$) be the 
height of $z\in \Kbar$ relative to the adelic generalized Mandelbrot set
$\M_{\bfa} := \prod_{v\in \Omega_K} M_{\bfa,v}$ (respectively, $\M_{\bfb}$) 
defined as in Section~\ref{preliminaries} (see also Remark~\ref{remark:second important remark}). Note that if $\lambda \in \Kbar$
is a parameter such that $\bfa(\l)$ 
(and $\bfb(\l)$) is torsion for $\Phi^\l$ then
$h_{\M_{\bfa}}(\lambda)=0$ by Corollary~\ref{second important remark}.  
So, we may apply the
equidistribution result from Theorem~\ref{thm:equidistribution}
and conclude that $M_{\bfa,v}=M_{\bfb,v}$ for each place $v\in \Omega_K$.
Indeed, we know that there exists an infinite sequence
$\{\l_n\}_{n\in\N}$ of distinct numbers $\l\in\Kbar$ such that both
$\bfa(\l)$ and $\bfb(\l)$ are torsion points for $\Phi^\l$. So, for each $n\in\N$, we may
take $S_n$ be the union of the sets of Galois conjugates for $\l_m$
for all $1\le m\le n$. Clearly $\#S_n\to\infty$ as $n\to\infty$, and
also each $S_n$ is $\Gal(\Kbar/K)$-invariant. Finally,
$h_{\M_{\bfa}}(S_n)=h_{\M_{\bfb}}(S_n)=0$ for all $n\in\N$, and thus
Theorem~\ref{thm:equidistribution} applies in this case. We obtain
that $\mu_{\M_{\bfa}}=\mu_{\M_{\bfb}}$ and since they are both
supported on $\M_{\bfa}$ (resp. $\M_{\bfb}$), we also get that
$\M_{\bfa}=\M_{\bfb}$. Consequently, the two height functions
$h_{\M_{\bfa}}$ and $h_{\M_{\bfb}}$ are equal. 
Using Corollary~\ref{second important remark} again, we conclude that for each $\l\in\Kbar$ we have that $\hhat_{\Phi^\l}(\bfa(\l))=0$ if and only if $\hhat_{\Phi^\l}(\bfb(\l))=0$. Hence $\bfa(\l)\in\Phi^\l_{\tor}$ if and only if $\bfb(\l)\in\Phi^\l_{\tor}$.

Finally, knowing that $\M_{\bfa}=\M_{\bfb}$ we obtain that the capacities of the corresponding generalized Mandelbrot sets are equal to each other, i.e.,
$$|C_{\bfa}|_v^{\frac{1}{\deg(\bfa)}}=|C_\bfb|_v^{\frac{1}{\deg(\bfb)}}\text{ for each place }v.$$
We let $U:=C_\bfa^{\deg(\bfb)}/C_\bfb^{\deg(\bfa)}$. Then for each $v\in\Omega_K$ we know that $|U|_v=1$. Since the only elements of $K$ which are units for all places $v\in\Omega_K$ are the ones living in $\Fqbar$, we conclude that $U\in\Fqbar$, as desired.


\section{Proof of our main results}
\label{conclusion}


Theorem~\ref{nice main result} follows from Theorem~\ref{main result}.
\begin{proof}[Proof of Theorem~\ref{nice main result}.]
If either $\bfa$ or $\bfb$ is a torsion point for  $\Phi$, then we
are done. So, from now on, assume that neither $\bfa$ nor
$\bfb$ is torsion for  $\Phi$. 

We now consider the given Drinfeld module
$\Phi:\Fq[t]\lra\End_{L}(\bG_a)$ as a Drinfeld module over $L = K(z)$
the rational function field with constant field $K.$ 
We let  $\Omega_{L/K}$ be the set of all places of the
function field $L/K$, and let $\hhat_\Phi$ be the canonical height for
the Drinfeld module $\Phi$ over $L.$ 
There are two cases.

{\bf Case 1.} Each polynomial $g_i$ is constant.

In this case, we have that $\Phi_{\tor}^\l = \Phi_{\tor} \subset
\Kbar$. Since $\bfa$  and $\bfb$ are not torsion for $\Phi$, 
we conclude that $\bfa$ (and similarly $\bfb$) is not a constant
polynomial, otherwise $\bfa \notin \Phi_{\tor}^\l$
($\bfb\notin\Phi_{\tor}^\l$, respectively) for all $\l\in \Kbar.$ This
contradicts our hypothesis that there exist infinitely many 
$\l\in\Kbar$ such that $\bfa(\l)$ and $\bfb(\l)$ are torsion for
$\Phi^\l$. Therefore, according to \cite[Theorem 4.15 (a)]{mw2}, there
exists $w\in\Omega_{L/K}$ such that $\hhat_{\Phi,w}(\bfa)>0$ (a
similar statement holds for $\bfb$). 

{\bf Case 2.} There exists at least one polynomial $g_i$ which is not constant.

In this case, there exists at least one place $v\in\Omega_{L/K}$ such
that for some $i=1,\dots, r-1$, we have $|g_i|_v>1$ and thus
\cite[Theorem 4.15 (b)]{mw2} yields that there exists some place
$w\in\Omega_{L/K}$ such that  $\hhat_{\Phi,w}(\bfa)>0$. A similar
statement holds for $\bfb$ since we know that neither $\bfa$ nor
$\bfb$ is torsion for $\Phi$. 

Therefore we know that there exists some place $w\in\Omega_{L/K}$ such that $\hhat_{\Phi,w}(\bfa)>0$, i.e., 
\begin{equation}
\label{w going to infinity}
|\Phi_{t^n}(\bfa)|_w\to\infty\text{ as $n\to\infty$.}
\end{equation}  
However, since $\bfa\in K[z]$ and  each $g_i\in K[z]$ and also
$t\in K$ we obtain that for each place $v\in\Omega_{L/K}$
\emph{except} the place at infinity, the iterates
$\Phi_{t^n}(\bfa)$ are all $v$-integral. Hence it must be that for
the place $w$ at infinity of the function field $K(z)$ we have that
\eqref{w going to infinity} holds, i.e.
$$\deg_z\left(\Phi_{t^n}(\bfa)\right)\to\infty\text{ as $n\to\infty$.}$$
A similar argument yields that also
$\deg_z\left(\Phi_{t^n}(\bfb)\right)\to\infty$ as
$n\to\infty$. Hence, at the expense of replacing both $\bfa$ and
$\bfb$ by $\Phi_{t^n}(\bfa)$ respectively $\Phi_{t^n}(\bfb)$
(for a sufficiently large integer $n$), we may achieve that inequality
\eqref{condition on degrees of starting points} is
satisfied. Moreover, note that for each $\l$, we have that $\bfa(\l)$
(or $\bfb(\l)$) is a torsion point for $\Phi^\l$ if and only if
$\Phi^\l_{t^n}(\bfa(\l))$ (respectively $\Phi^\l_{t^n}(\bfb(\l))$) is
a torsion point for $\Phi^\l$. Theorem~\ref{main result} yields that
for each $\l\in\Kbar$ we have that  $\Phi^\l_{t^n}(\bfa(\l))$ (and
thus $\bfa(\l)$) is a torsion point for $\Phi^\l$ if and only if
$\Phi^\l_{t^n}(\bfb(\l))$ (and thus $\bfb(\l)$) is a torsion point for
$\Phi^\l$.  
\end{proof}

In the following, for a nonzero element $f\in \Fq[t]$ and a Drinfeld
module $\phi$ we denote the submodule of $f$-torsion  by
$\phi[f]$ as usual. 
\begin{proof}[Proof of Theorem~\ref{nice corollary}.]
If either $\bfa$ or $\bfb$ equals $0$, then clearly $\bfa$ (or $\bfb$) is always torsion and the conclusion is immediate. So, assume now that both $\bfa$ and $\bfb$ are nonzero.

Assume first that $\bfa$ and $\bfb$ are $\Fq$-linearly dependent.  We
note that for each positive integer $n$ we have then that
$\Phi_{t^n}(\bfa)$ is a  polynomial of degree $q^{rn}$ in $z$. On the
other hand, for any two distinct, monic, irreducible polynomials
$f,g\in\Fp[t]$ we have $\Phi[f]\cap \Phi[g] = \{0\}.$ 
Therefore by solving $\Phi_{f}^\l(\bfa)=0$ for various distinct,
monic, irreducible polynomials $f\in\Fq[t]$ we find that there exist
infinitely many $\l\in\Kbar$ such that $\bfa\in\Phi^\l_{\tor}.$
On the other hand, for any $c\in \Fq$ and $a\in \Fq[t]$ we have that
$c\Phi_a^\l(x)=\Phi_a^\l(cx)$. From this, it is easy to see that for
each $\l\in\Kbar$ we have that $\bfa$ is a torsion point if and only
if $\bfb = c \bfa$ ($ c\in \Fq$) is a torsion point for $\Phi^\l$. 
So, indeed there exist infinitely many $\l\in\Kbar$ such that $\bfa,\bfb\in\Phi^\l_{\tor}$.

Now for the converse, we note that
$$\bfa_1:=\Phi_{t^2}(\bfa)=\Phi_t\left(\bfa^q\cdot z+\left(t\bfa + \bfa^{q^r}\right)\right)$$
and
$$\bfb_1:=\Phi_{t^2}(\bfb)=\Phi_t\left(\bfb^q\cdot z +\left(t\bfb+\bfb^{q^r}\right)\right)$$
are both  polynomials in $z$ of same degree $q^r>1$. Furthermore their leading coefficients are $\bfa^{q^{r+1}}$, respectively $\bfb^{q^{r+1}}$. Therefore inequality \eqref{condition on degrees of starting points} from Theorem~\ref{main result} is satisfied and thus we conclude that $\bfa/\bfb\in\Fpbar$. On the other hand we know that $\bfa,\bfb\in K$ and $\Fq$ is algebraically closed in $K$; hence $\bfa/\bfb\in\Fq$ as desired.
\end{proof}

Similarly to Theorem~\ref{nice corollary} we can prove Theorem~\ref{nice corollary 2}.
\begin{proof}[Proof of Theorem~\ref{nice corollary 2}.]
Let $s$ be a positive integer such that $\bfa,\bfb\in\Fqs(t)$ and let $\Omega_s:=\Omega_{\Fqs(t)}$. Arguing as in the proof of Theorem~\ref{nice corollary} we obtain that if $\bfa$ and $\bfb$ are $\Fq$-linearly dependent, then there exist infinitely many $\l\in\overline{\Fp(t)}$ such that $\bfa,\bfb\in\Phi^\l_{\tor}$. 

Assume now that there exist infinitely many $\l\in\overline{\Fq(t)}$
such that $\bfa,\bfb\in\Phi_{\tor}^\l$. In addition, we may assume
both $\bfa$ and $\bfb$ are nonzero. The proof of Theorem~\ref{nice
  corollary}  shows that $\bfa/\bfb\in\Fqbar$.
In addition, Theorem~\ref{main result} yields that for each $\l\in\overline{\Fq(t)}$, we have that $\bfa\in\Phi^\l_{\tor}$ if and only if $\bfb\in\Phi^\l_{\tor}$. In order to finish the proof of Theorem~\ref{nice corollary 2} we will use both consequences of Theorem~\ref{main result} stated above.

We let $\gamma:=\bfb/\bfa\in\Fqs$, and assume $\g\notin\Fq$. Let
$\l_0\in \overline{\Fq(t)}$ such that $\Phi^{\l_0}_t(\bfa)=0$. Then, we have
\begin{equation}
\label{eq: lambda0} \l_0 \bfa^q = - t \bfa - \bfa^{q^r}. 
\end{equation}
We will show that $\bfb\notin\Phi^{\l_0}_{\tor}$ which yields a
contradiction to the conclusion of Theorem~\ref{main result}.
Before we proceed, we note that 
\begin{align}
\Phi_t^{\l_0}(\bfb) & = \g t \bfa + \g^q \l_0 \bfa^q+ \g^{q^2}
\bfa^{q^2} \notag \\
  & = (\g - \g^q) t \bfa + (\g^{q^r} - \g^q) \bfa^{q^r}
  \;\text{by~\eqref{eq: lambda0}}
\label{eq: phi of b} .
  \end{align}
In the following, we denote by $|\cdot |_\infty$  the  absolute value corresponding  to the
unique place of $\Fqs(t)$ where $t$ is not integral.  
We split our analysis into two cases: 
\medskip
\\
{\bf Case 1.} $\bfa\in\Fqs$.

By~\eqref{eq: lambda0}  we have that $|\l_0|_\infty = |t|_\infty$ and
it follows from~\eqref{eq: phi of b} that 
$$
|\Phi_t^{\l_0}(\bfb) |_\infty = |(\g - \g^q) t \bfa + (\g^{q^r} - \g^q)
\bfa^{q^r} |_\infty = |t|_\infty > \max\{1, |t|_\infty^{1/(q^r - 1)},
|\l_0|_\infty^{1/(q^r - q)}\}. 
$$   
Using \eqref{definition N_v}, we conclude that
$\left|\Phi^{\l_0}_{t^n}(\bfb)\right|_\infty \to \infty$ as $n\to\infty$ and hence
$\bfb\not\in \Phi^{\l_0}_{\tor}$ as desired. 
\medskip
\\
{\bf Case 2.} $\bfa\notin\Fqs$.

In this case there exists a place $v\in\Omega_s$ such that
$|\bfa|_v>1$.  Note that it is not possible to  have $|t \bfa|_v \ge
|\bfa|^{q^r}$ for otherwise also  $|t|_v > 1$. This implies that 
$|\cdot|_v = |\cdot|_\infty$ and $|t|_\infty \ge
|\bfa|_\infty^{q^r-1}.$ But this is impossible in $\Fqs(t).$ Hence,
$|t \bfa|_v <  |\bfa|_v^{q^r}$ and so, $|\l_0|_v=|\bfa|_v^{q^r-q}$. 

There are two possibilities now.

{\bf Case 2a.} $\g^{q^r} - \g^q \ne 0$. 

Consequently, we have
$$|\Phi_t^{\l_0}(\bfb)|_v = \max\{|t \bfa|_v,  |\bfa|_v^{q^r}\} = |\bfa|_v^{q^r}.$$ 
It follows that
\begin{itemize}
\item $|t|_v^{1/(q^r -1)} \le |\bfa|_v$ ,
\item $|\l_0|_v^{1/(q^r - q)} = |\bfa|_v$  and
  \item $|\Phi_t^{\l_0}(\bfb)|_v = |\bfa|_v^{q^r} > |\bfa|_v =
    \max\{1, |t|_v^{1/(q^r - 1)}, |\l_0|_v^{1/(q^r - q)}\}.$  
  \end{itemize}
 Using  \eqref{definition N_v} again, we conclude that 
$\left|\Phi^{\l_0}_{t^n}(\bfb)\right|_ v\to \infty$ as $n\to\infty$
and $\bfb\not\in \Phi^{\l_0}_{\tor}$. 

{\bf Case 2b.} $\g^{q^r}=\g^q$.

Since we assumed that $\g\notin\Fq$, then we obtain that $r\ge 3$.
By~\eqref{eq: phi of b} and the assumption in {\bf Case 2b} we have that 
\begin{eqnarray*}
 \Phi^{\l_0}_t(\bfb) & = (\g-\g^q)t\bfa
\end{eqnarray*}
If $|t|_v>1$ (i.e., $v$ is the place $\infty$) then $|\Phi_t^{\l_0}(\bfb)|_v>|\bfa|_v$ and so, again \eqref{definition N_v} can be used to infer that $|\Phi_{t^n}^{\l_0}(\bfb)|_v\to\infty$ since 
$$|\Phi_t^{\l_0}(\bfb)|_v>|\bfa|_v= \max\left\{|t|_v^{1/(q^r-1)}, |\l_0|^{1/(q^r-q)}\right\}.$$

Assume now that $|t|_v\le 1$. Using the assumption on $\g$ from {\bf Case 2b} we get
\begin{align*}
 \Phi_{t^2}^{\l_0}(\bfb)
& = (\g-\g^q)t^2\bfa + (\g-\g^q)^qt^q\bfa^q\l_0 + (\g-\g^q)^{q^r}t^{q^r}\bfa^{q^r}\\
& = (\g-\g^q)t^2\bfa + \left(\g^q-\g^{q^2}\right)t^q \bfa^q\l_0 + \left(\g^q-\g^{q^2}\right) t^{q^r}\bfa^{q^r}\\
& = (\g-\g^q)t^2\bfa + \left(\g^q-\g^{q^2}\right)\cdot \left(t^q \bfa^q\l_0+ t^{q^r}\bfa^{q^r}\right)\\
& = (\g-\g^q)t^2\bfa + \left(\g^q-\g^{q^2}\right)\cdot \left(-t^{q+1}\bfa-t^q\bfa^{q^r}+t^{q^r} \bfa^{q^r}\right)\text{ using \eqref{eq: lambda0}}\\
& = (\g-\g^q)t^2\bfa + \left(\g^q-\g^{q^2}\right)\cdot \left(-t^{q+1}\bfa + \bfa^{q^r}\cdot \left(t^{q^r}-t^q\right)\right).
\end{align*}
Since $|\bfa|_v>1$ and $v\in\Omega_s$ we conclude that
\begin{equation}
\label{the valuation is not too large}
\left|t^{q^r}-t^q\right|_v=\left|t^{q^{r-1}}-t\right|_v^q \ge |\bfa|_v^{-q}
\end{equation}
since $t^{q^{r-1}}-t$ is a separable polynomial and thus it is either a $v$-adic unit or has the $v$-adic absolute value equal to that of a uniformizer of $v$ in $\Fqs(t)$ (note the assumption that $|t|_v\le 1$). So, using \eqref{the valuation is not too large} we get that
\begin{align*}
 \left|\bfa^{q^r}\cdot \left(t^{q^r}-t^q\right)\right|_v
& \ge |\bfa|_v^{q^r-q}\\
& \ge |\bfa|_v^{q^3-q}\text{ because $r\ge 3$ in {\bf Case 2b}}\\
& > |\bfa|_v\text{ because $q\ge 2$}\\
& \ge \left|t^{q+1}\bfa\right|_v\text{ because $|t|_v\le 1$ by our assumption.}
\end{align*}
Therefore (using also that $\g\notin\Fq$ and thus $\g^q-\g^{q^2}\ne 0$)
\begin{equation}
\label{valuation of the second part}
\left|\left(\g^q-\g^{q^2}\right)\cdot \left(-t^{q+1}\bfa + \bfa^{q^r}\cdot \left(t^{q^r}-t^q\right)\right)\right|_v\ge |\bfa|_v^{q^r-q}.
\end{equation}
On the other hand
\begin{equation}
\label{valuation of the first part}
\left|(\g-\g^2)t^2\bfa\right|_v\le |\bfa|_v.
\end{equation}
because $|t|_v\le 1$. Using \eqref{valuation of the second part} and \eqref{valuation of the first part} coupled with the fact that
$$q^r-q>1\text{ since $r\ge 3$ and $q\ge 2$}$$
we conclude that
$$\left|\Phi_{t^2}^{\l_0}(\bfb)\right|_v\ge |\bfa|_v^{q^r-q}>|\bfa|_v= \max\left\{|t|_v^{1/(q^r-1)}, |\l_0|^{1/(q^r-q)}\right\}.$$
Again using \eqref{definition N_v} yields that $|\Phi^{\l_0}_{t^n}(\bfb)|_v\to\infty$ as $n\to\infty$ and thus $\bfb\notin\Phi^{\l_0}_{\tor}$.

In conclusion, assuming that $\g\notin\Fq$ yields in each case a contradiction; this finishes the proof of Theorem~\ref{nice corollary 2}.
\end{proof}




\end{document}